\documentclass[11pt]{amsart}
\usepackage{amsfonts,amssymb,amsmath,amsthm,mathrsfs}
\usepackage{url}
\usepackage{enumerate}

\usepackage[pdftex, pagebackref, bookmarksnumbered, bookmarksopen, colorlinks, citecolor=blue, linkcolor=blue]{hyperref}

\newtheorem{theorem}{Theorem}[section]
\newtheorem{lemma}[theorem]{Lemma}

\newtheorem{corollary}[theorem]{Corollary}
\theoremstyle{definition}

\newtheorem{remark}[theorem]{Remark}
\numberwithin{equation}{section}

\author[G. Hu]{Guoen Hu}
\address{Guoen Hu: School of Applied Mathematics, Beijing Normal University, Zhuhai 519087,
P. R. China}
\email{huguoen@yahoo.com}
\thanks{The research of the first author was supported by NSFC (No. 11871108). The research of second author was supported by China Postdoctoral Science Foundation (Nos. 2017M621253, 2018T110279), National Natural Science Foundation of China (No. 11801118) and the Fundamental Research Funds for the Central Universities. The research of the third author was supported
partly by NSFC (Nos. 11471041, 11671039, 11871101) and NSFC-DFG (No. 11761131002).}
\author[X. Lai]{Xudong Lai}
\address{Xudong Lai: Institute for Advanced Study in Mathematics, Harbin Institute of Technology, Harbin, 150001, People's Republic of China}
\email{xudonglai@hit.edu.cn\ xudonglai@mail.bnu.edu.cn}
\author[Q. Xue]{Qingying Xue}
\address{Qingying Xue: School of Mathematics, Beijing Normal university, Beijing 100875, People's Republic of China}
\email{qyxue@bnu.edu.cn}
\thanks{Xudong Lai is the corresponding author.}

\begin{document}

\title[compositions of rough operators]{On the composition of rough singular integral operators}

\begin{abstract}
In this paper, we investigate the behavior of the bounds of the composition for rough singular integral operators on the weighted space. More precisely,
we obtain the quantitative weighted bounds of the composite operator for two singular integral operators with rough homogeneous kernels on $L^p(\mathbb{R}^d,\,w)$, $p\in (1,\,\infty)$, which is smaller than the product of the quantitative weighted bounds for these two rough singular integral operators. Moreover, at the endpoint $p=1$, the $L\log L$ weighted weak type bound is also obtained, which has interests of its own in the theory of rough singular integral even in the unweighted case.
\end{abstract}
\maketitle
\section{Introduction}
This paper will be devoted to study the quantitative weighted bounds for the composition of rough singular integral operators. The theory of Calder\'on-Zygmund singular integral operator, which origins from the pioneering work of Calder\'on and Zygmund \cite{cz1} in 1950s, has been developed extensively in the last sixty years (see for example the recently exposition \cite{Gra249},\cite{Gra250}).

The composition of singular integral operators arise typically  in the algebra of singular integral (see \cite{CZ56},\cite{Cal67},\cite{Cal68}) and  the non-coercive
boundary-value  problems  for elliptic equations (see \cite{phst},\cite{NRSW15}). In the past decades, considerable attention has been paid to the composition of singular integral operators.
We refer the reader to see the work in \cite{str,phst,Chr88,ober,cana,NRSW15} and the references therein. This paper will be devoted to study the composition of the singular integral operator $T_{\Omega}$ with a rough convolution type kernel. Recall that
 $T_\Omega$ is defined by
\begin{eqnarray}\label{eq1.1}{T}_{\Omega}f(x)={\rm p.\,v.}\int_{\mathbb{R}^d} \frac
{\Omega(x-y)}{|x-y|^d}f(y)dy,\end{eqnarray}
where $\Omega$ is
homogeneous of degree zero, integrable and has mean value zero on
the unit sphere ${S}^{d-1}$. This operator was introduced by
Calder\'on and Zygmund \cite{cz1}, and then studied by many authors
in the last sixty years (see e.g. \cite {cz2}, \cite{co}, \cite{rw}, \cite{chr2}, \cite{gs}, \cite{fp}, \cite{se}).
The composite operator $T_{\Omega_1}T_{\Omega_2}$ has been first appeared in the work of  Calder\'on and Zygmund \cite{CZ56} where the algebra of singular integrals was studied.
However in this paper, we will study other properties
 of the composite operator $T_{\Omega_1}T_{\Omega_2}$.
Our starting points of this paper are as follows:
\begin{enumerate}[\quad (i).]
\item Calder\'on and Zygmund \cite {cz2} proved that $T_{\Omega}$ is bounded on $L^p(\mathbb{R}^d)$ if $p\in (1,\,\infty)$ for rough kernel $\Omega$. It is trivial to see that the composite operator $T_{\Omega_1}T_{\Omega_2}$ is bounded on $L^p(\mathbb{R}^d)$ for $p\in(1,\infty)$. At the endpoint $p=1$, it was quite later that Seeger \cite{se} showed $T_\Omega$ is of weak type (1,1) by means of  some deep idea of geometric microlocal decomposition and the Fourier transform. Nevertheless, no proper weak type estimate of $T_{\Omega_1}T_{\Omega_2}$ was known prior to this article when both $\Omega_1$ and $\Omega_2$ are rough kernels. In this paper, we will prove that $T_{\Omega_1}T_{\Omega_2}$ satisfies the $L\log L$ weak type estimate.
\item Recently there are numerous work related to seek the optimal quantitative weighted bound for singular integral operator (see e.g. \cite{ccdo,ler4,ler5,ler6,dhl,hlp,hrt,hp,bb,hu1,hu3}). Motivated by this, our interests are focused on the behavior of the quantitative weighted bound for $T_{\Omega_1}T_{\Omega_2}$ compared to that of single singular integral. We show that the quantitative weighted bound of $T_{\Omega_1}T_{\Omega_2}$ is smaller than the products of that of $T_{\Omega_1}$ and $T_{\Omega_2}$, which has interests of its own.
\end{enumerate}

We summary our main results as follows.

\begin{theorem}\label{thm1.1}
Let $\Omega_1$, $\Omega_2$ be homogeneous of degree zero, have mean value zero and $\Omega_1,\,\Omega_2\in L^{\infty}({S}^{d-1})$. Then for $p\in (1,\,\infty)$ and $w\in A_{p}(\mathbb{R}^d)$,
\begin{eqnarray*}
\|T_{\Omega_1}T_{\Omega_2}f\|_{L^p(\mathbb{R}^d,\,w)}&\lesssim &[w]_{A_p}^{\frac{1}{p}} \big([w]_{A_{\infty}}^{\frac{1}{p'}}+[\sigma]_{A_{\infty}}^{\frac{1}{p}}\big)\big([\sigma]_{A_{\infty}}+[w]_{A_{\infty}}\big)\\
&&\times \min\big\{[\sigma]_{A_{\infty}},\,[w]_{A_{\infty}}\}\|f\|_{L^{p}(\mathbb{R}^d,\,w)},
\end{eqnarray*}
where $p'=p/(p-1)$, $\sigma=w^{-1/(p-1)}$, and the precise definitions of $A_p(\mathbb{R}^d)$ weight and $A_p$ constants are listed in Section 2.
\end{theorem}
\begin{remark}It is unknown whether the above quantitative weighted bound is optimal. However, from the recent result of  Hyt\"onen, Roncal, and  Tapiola \cite{hrt}: if $\Omega\in L^{\infty}({S}^{d-1})$, then for $p\in (1,\,\infty)$ and $w\in A_p(\mathbb{R}^d)$,
\begin{equation*}\|T_{\Omega}f\|_{L^p(\mathbb{R}^d,\,w)}\lesssim [w]_{A_p}^{\frac{1}{p}} \big([w]_{A_{\infty}}^{\frac{1}{p'}}+[\sigma]_{A_{\infty}}^{\frac{1}{p}}\big)\big([\sigma]_{A_{\infty}}+[w]_{A_{\infty}}\big)
\|f\|_{L^p(\mathbb{R}^d,\,w)},
\end{equation*}
in which the quantitative weighted bound was improved later by Li, P\'erez,  Rivera-Rios and Roncal \cite{lpr} as follows,
\begin{equation}\label{eq1.3} [w]_{A_p}^{\frac{1}{p}} \big([w]_{A_{\infty}}^{\frac{1}{p'}}+[\sigma]_{A_{\infty}}^{\frac{1}{p}}\big)\min\{[\sigma]_{A_{\infty}},\,[w]_{A_{\infty}}\},
\end{equation}
we can see that the quantitative weighted bound of $T_{\Omega_1}T_{\Omega_2}$ in Theorem \ref{thm1.1} is smaller than the product of the quantitative weighted bounds of $T_{\Omega_1}$ and $T_{\Omega_2}$ in (\ref{eq1.3}). In fact, for $p\in (1,\,\infty)$ and $w\in A_{p}(\mathbb{R}^d)$, by some elementary computation,
$$\max\{[w]_{A_{\infty}},\,[\sigma]_{A_{\infty}}\}\leq [w]_{A_p}^{\frac{1}{p}} \big([w]_{A_{\infty}}^{\frac{1}{p'}}+[\sigma]_{A_{\infty}}^{\frac{1}{p}}\big),$$
which  easily implies our desired estimate.
\end{remark}
\begin{theorem}\label{thm1.2}
Let $\Omega_1$, $\Omega_2$ be homogeneous of degree zero, have mean value zero and $\Omega_1,\,\Omega_2\in L^{\infty}({S}^{d-1})$. Then for $w\in A_1(\mathbb{R}^d)$ and $\lambda>0$,
\begin{eqnarray*}
&&w(\{x\in\mathbb{R}^d:\,|T_{\Omega_1}T_{\Omega_2}f(x)|>\lambda\})\\
&&\quad\quad\lesssim [w]_{A_1}[w]_{A_{\infty}}^2\log ({\rm e}+[w]_{A_{\infty}})\int_{\mathbb{R}^d}\frac{|f(x)|}{\lambda}\log\Big({\rm e}+
\frac{|f(x)|}{\lambda}\Big)w(x)dx.
\end{eqnarray*}
\end{theorem}
\begin{remark}
To the best knowledge of the author, the $L\log L$ weak type estimate in Theorem \ref{thm1.2} is new even in the unweighted case. We do not know whether this kind of $L\log L$ weak type estimate is optimal, but this estimate has no hope to be improved to the weak type (1,1) estimate even in the case $\Omega_1,\Omega_2\in C^\infty(S^{d-1})$. In fact, it was shown by Phong and Stein \cite{phst} that in general the composite operator $T_{\Omega_1}T_{\Omega_2}$ is not of weak type (1,1). More ever, the authors of \cite{phst} gave a necessary and sufficient condition such that the composite operator is of weak type (1,1). If $\Omega_1,\Omega_2\in C^\infty(S^{d-1})$, then by \cite[Proposition 2.4.8]{Gra249}, the symbols of $T_{\Omega_1}$ and  $T_{\Omega_2}$ (thus is $\mathcal{F}[p.v.\Omega_1(\cdot)/|\cdot|^d]$ and $\mathcal{F}[p.v.\Omega_2(\cdot)/|\cdot|^d]$, where $\mathcal{F}[f]$ denote the Fourier transform of $f$) are $C^\infty(\mathbb{R}^d\setminus\{0\})$. By check the necessary and sufficient condition in \cite[Theorem 1]{phst}, we may show that $T_{\Omega_1}T_{\Omega_2}$ is not of weak type (1,1).
\end{remark}

Previous results of quantitative weighted bounds for the composite operator is only known for the smooth singular integral operators, we refer to see \cite{bb},\cite{hu1} and \cite{hu3}. It should be pointed out that the argument for the smooth singular integral operators used in \cite{bb,hu1,hu3} essentially relies on the smooth condition of the kernel. Our strategy in this paper is to establish a decomposition of the composite operator by representing it as two operators which may have different kinds of bilinear sparse dominations: $(L(\log L)^{\beta},\,L^{r})$ and $(L^1,\,L^{r})$ type respectively (see Corollary \ref{c41}). This decomposition is done basing on the weak type estimates of the grand maximal operator $\mathscr{M}_{T_\Omega,r}$ and $T_\Omega$. In addition, we also show that the $(L(\log L)^{\beta},\,L^{r})$ type sparse domination could be applied to the operator that is of $(L(\log L)^{\beta}$ weak type to get quantitative weighted bounds. Our main arguments (see Sections 3 and 4) presented in this paper are stated in the abstract setting which have interest of its own. By applying them to the composite operator $T_{\Omega_1}T_{\Omega_2}$, we may get our main theorems.

This paper is organized as follows. In Section 2, we give some notation and lemmas. In Section 3, we will establish an quantitative weighted weak type estimate for the operator which enjoys a bilinear sparse domination. Section 4 is devoted to give a decomposition of the composite operator. Finally as applications of the arguments in Sections 3 and 4, the proof of our main theorems are given in Section 5.
\vskip0.24cm

\section{Preliminary}
In this paper, we will work on $\mathbb{R}^d$, $d\geq 2$. $C$ always denotes a
positive constant that is independent of the main parameters
involved but whose value may differ from line to line. We use the
symbol $A\lesssim B$ to denote that there exists a positive constant
$C$ such that $A\le CB$.  Specially, we use $A\lesssim_{d,p} B$ to denote that there exists a positive constant
$C$ depending only on $d,p$ such that $A\le CB$. Constant with subscript such as $c_1$, does not change in different occurrences.

For any set $E\subset\mathbb{R}^d$,
$\chi_E$ denotes its characteristic function.  For a cube
$Q\subset\mathbb{R}^d$ and $\lambda\in(0,\,\infty)$, we use $\ell(Q)$ (${\rm diam}Q$) to denote the side length (diameter) of $Q$, and
$\lambda Q$ to denote the cube with the same center as $Q$ and whose
side length is $\lambda$ times that of $Q$.  For a fixed cube $Q$, denote by $\mathcal{D}(Q)$ the set of dyadic cubes with respect to $Q$, that is, the cubes from $\mathcal{D}(Q)$ are formed by repeating subdivision of $Q$ and each of descendants into $2^d$ congruent subcubes.

For  $\beta\in [0,\,\infty)$,  cube $Q\subset \mathbb{R}^d$ and a suitable function $g$, $\|g\|_{L(\log L)^{\beta},\,Q}$ is the norm defined by
$$\|g\|_{L(\log L)^{\beta},\,Q}=\inf\Big\{\lambda>0:\,\frac{1}{|Q|}\int_{Q}\frac{|g(y)|}{\lambda}\log^{\beta}\Big({\rm e}+\frac{|g(y)|}{\lambda}\Big)dy\leq 1\Big\}.$$
$\langle |f|\rangle_{Q}$ denotes the mean value of $|f|$ on $Q$ and $\langle |g|\rangle_{Q, r}=\big(\langle|g|^r\rangle_{Q}\big)^{\frac{1}{r}}.$ We denote $\|g\|_{L(\log L)^{0},\,Q}$ by $\langle |g|\rangle_{Q}$.
Let $M_{\beta}$ be the maximal operator defined by
$$M_{\beta}f(x)=\big[M(|f|^{\beta})(x)\big]^{\frac{1}{\beta}},$$
where $M$ is the Hardy-Littlewood maximal operator, and $M_{L(\log L)^{\beta}}$ be the maximal operator defined by
 $$M_{L(\log L)^{\beta}}g(x)=\sup_{Q\ni x}\|g\|_{L(\log L)^{\beta},\,Q}.$$
For simplicity, we denote $M_{L(\log L)^{1}}$ by $M_{L\log L}$. It is well known that $M_{L(\log L)^{\beta}}$ is bounded on $L^p(\mathbb{R}^n)$ for all $p\in (1,\,\infty)$, and for any $\lambda>0$,
\begin{equation}\label{eq1.7}\big|\{x\in\mathbb{R}^d:\,M_{L(\log L)^{\beta}}g(x)>\lambda\}\big|\lesssim \int_{\mathbb{R}^d}\frac{|g(x)|}{\lambda}\log^{\beta} \Big({\rm e}+\frac{|g(x)|}{\lambda}\Big)dx.
\end{equation}

Let $w$ be a nonnegative, locally integrable function on $\mathbb{R}^d$. We say that   $w\in A_{p}(\mathbb{R}^d)$ if the $A_p$ constant $[w]_{A_p}$ is finite, where
$$[w]_{A_p}:=\sup_{Q}\Big(\frac{1}{|Q|}\int_Qw(x)dx\Big)\Big(\frac{1}{|Q|}\int_{Q}w^{-\frac{1}{p-1}}(x)dx\Big)^{p-1},\,\,\,p\in (1,\,\infty),$$
the  supremum is taken over all cubes in $\mathbb{R}^d$, and the $A_1$ constant is defined by
$$[w]_{A_1}:=\sup_{x\in\mathbb{R}^d}\frac{Mw(x)}{w(x)}.$$
A weight $u\in A_{\infty}(\mathbb{R}^d)=\cup_{p\geq 1}A_p(\mathbb{R}^d)$. We use the following definition of the $A_{\infty}$ constant of $u$ (see e.g. \cite{wil})
$$[u]_{A_{\infty}}=\sup_{Q\subset \mathbb{R}^d}\frac{1}{u(Q)}\int_{Q}M(u\chi_Q)(x)dx.$$

As usual, by a general dyadic grid $\mathscr{D}$,  we mean a collection of cubes with the following properties: (i) for any cube $Q\in \mathscr{D}$, its side length $\ell(Q)$ is of the form $2^k$ for some $k\in \mathbb{Z}$; (ii) for any cubes $Q_1,\,Q_2\in \mathscr{D}$, $Q_1\cap Q_2\in\{Q_1,\,Q_2,\,\emptyset\}$; (iii) for each $k\in \mathbb{Z}$, the cubes of side length $2^k$ form a partition of $\mathbb{R}^d$.

Let $\eta\in (0,\,1)$ and $\mathcal{S}=\{Q_j\}$ be a family of cubes. We say that $\mathcal{S}$ is $\eta$-sparse,  if for each fixed $Q\in \mathcal{S}$, there exists a measurable subset $E_Q\subset Q$, such that $|E_Q|\geq \eta|Q|$ and $E_{Q}$'s are pairwise disjoint.
Associated with  the sparse family $\mathcal{S}$  and constants $\beta\in[0,\,\infty)$ and $r\in [1,\,\infty)$, we define the bilinear sparse operator $\mathcal{A}_{\mathcal{S};\,L(\log L)^{\beta},\,L^{r}}$  by
$$\mathcal{A}_{\mathcal{S};\,L(\log L)^{\beta},L^{r}}(f,g)=\sum_{Q\in\mathcal{S}}|Q|\|f\|_{L(\log L)^{\beta},\,Q}\langle|g|\rangle_{Q,r}.$$
Also, we define  the operator $\mathcal{A}_{\mathcal{S},\,L^{r_1},\,L^{r_2}}$ by
$$\mathcal{A}_{\mathcal{S};\,L^{r_1},L^{r_2}}(f,g)=\sum_{Q\in\mathcal{S}}|Q|\langle|f|\rangle_{Q,r_1}\langle |g|\rangle_{Q,\,r_2}.$$

Let $T$ be a sublinear operator acting on $\cup_{p\geq 1}L^p(\mathbb{R}^d)$, $\beta,\,q\in (0,\,\infty)$. We say that $T$ enjoys a
$(L(\log L)^{\beta},\,L^{q})$-bilinear sparse domination with bound $A$, if for each  bounded  function $f$ with compact support, there exists a sparse family $\mathcal{S}$ of cubes, such that for all $g\in L_{{\rm loc}}^{q}(\mathbb{R}^d)$,
\begin{eqnarray}\label{gongshi3.1}\Big|\int_{\mathbb{R}^d}g(x)Tf(x)dx\Big|\le A \mathcal{A}_{\mathcal{S},\,L(\log L)^{\beta},\,L^{q}}(f,\,g).
\end{eqnarray}

We will use the following lemmas in our proof.
\begin{lemma}[see \cite{hp2}]\label{lem4.1}
Let $t\in (1,\,\infty)$. Then for $p\in (1,\,\infty)$ and weight $w$,
$$\|Mf\|_{L^{p'}(\mathbb{R}^d,\,(M_tw)^{1-p'})}\le c_dpt'^{\frac{1}{p'}}\|f\|_{L^{p'}(\mathbb{R}^d,\,w^{1-p'})}.$$
\end{lemma}

\begin{lemma}[see \cite{lpr} or \cite{ler6}]\label{lem4.2} Let $p\in (1,\,\infty)$ and $v$ be a weight. Let $S$ be the operator defined by $$S(h) =v^{-\frac{1}{p}}M(hv^{\frac{1}{p}})$$ and $R$ be the operator defined by
\begin{eqnarray}\label{eq4.1}
R(h) = \sum_{k=0}^{\infty}\frac{1}{2^k}\frac{S^kh}{\|S\|_{L^p(\mathbb{R}^d,\,v)\rightarrow L^p(\mathbb{R}^d,\,v)}^k}.\end{eqnarray}
Then for any $h\in L^p(\mathbb{R}^d,\,v)$,
\begin{itemize}\item[\rm (i)] $0\leq h\leq R(h)$,
\item[\rm (ii)] $\|R(h)\|_{L^p(\mathbb{R}^d,\,v)}\leq  2\|h\|_{L^p(\mathbb{R}^d,\,v)},$
\item[\rm (iii)] $R(h)v^{\frac{1}{p}}\in A_1(\mathbb{R}^d)$ with $[R(h)v^{\frac{1}{p}}]_{A_1}\leq c_dp'$. Furthermore, when $v = M_rw$
for some $r\in [1,\,\infty)$,  we also have that $[Rh]_{A_{\infty}} \leq c_dp'.$
\end{itemize}
\end{lemma}

\begin{lemma}[see \cite{hp}]\label{lem4.3} Let $w\in A_{\infty}(\mathbb{R}^d)$. Then for any cube $Q$ and $\delta\in (1,\,1+\frac{1}{2^{11+d}[w]_{A_{\infty}}}]$,
$$\Big(\frac{1}{|Q|}\int_Qw^{\delta}(x)dx\Big)^{\frac{1}{\delta}}\leq \frac{2}{|Q|}\int_{Q}w(x)dx.$$
\end{lemma}
\vskip0.24cm

\section{Endpoint estimates for sparse operators}

The main purpose of this section is to establish a weighted weak type endpoint estimate for the operator which enjoys $(L(\log)^{\beta},L^q)$-bilinear sparse domination. We begin with some lemmas.

\begin{lemma}\label{lem4.4}Let $\beta\in [0,\,\infty)$, $r\in [1,\,\infty)$ and $w$ be a weight. Then for any $t\in (1,\,\infty)$ and $p\in(1,\,r')$ such that
$t\frac{p'/r-1}{p'-1}>1$,
\begin{equation*}\mathcal{A}_{\mathcal{S},L(\log L)^{\beta},L^r}(f,\,g)\lesssim p'^{1+
\beta}\big(\frac{p'}{r}\big)'\big(t\frac{p'/r-1}{p'-1}\big)'^{\frac{1}{p'}}\|f\|_{L^p(\mathbb{R}^d,\,M_tw)}
\|g\|_{L^{p'}(\mathbb{R}^d,w^{1-p'})}.
\end{equation*}
\end{lemma}
\begin{proof}
Let $p\in (1,\,r')$, $f\in C^{\infty}_0(\mathbb{R}^d)$ with $\|f\|_{L^{p}(\mathbb{R}^d,\,M_tw)}=1$ and $Rf$ be the function defined by (\ref{eq4.1}). Recall that
$$\|f\|_{L(\log L)^{\beta},\,Q}\lesssim \Big(1 +\big(\frac{\beta}{s-1}\big)^{\beta}\Big)\Big(\frac{1}{|Q|}\int_{Q}|f(y)|^sdy\Big)^{\frac{1}{s}}.$$
Applying Lemma \ref{lem4.2} with $v=M_tw$ and Lemma \ref{lem4.3}, we then get that
\begin{eqnarray*}
\sum_{Q\in \mathcal{S}}\|f\|_{L(\log L)^{\beta},\,Q}\langle |g|\rangle_{Q,\,r}|Q|&\lesssim &s'^{\beta}
\sum_{Q\in\mathcal{S}}\langle |g|\rangle_{Q,\,r}\langle |f|\rangle_{Q,\,s}|Q|\\
&\lesssim&s'^{\beta}\sum_{Q\in\mathcal{S}}\langle |g|\rangle_{Q,\,r}\langle |Rf|\rangle_{Q,\,s}|Q|\\
&\lesssim&p'^{\beta}\sum_{Q\in\mathcal{S}}\langle |g|\rangle_{Q,\,r}\int_{Q}Rf(y)dy,\end{eqnarray*}
if we choose $s=1+\frac{1}{2^{11+d}[Rf]_{A_{\infty}}}$.
As in the proof of Lemma 4.1 in \cite{hp2}, we see that
\begin{eqnarray*}
\sum_{Q\in\mathcal{S}}\langle |g|\rangle_{Q,\,r}\int_{Q}Rf(x)dx&\lesssim& \sum_{Q\in \mathcal{S}}\inf_{y\in Q}M_rg(y)\int_{Q}Rf(x)dx\\
&\lesssim&[Rh]_{A_{\infty}}\int_{\mathbb{R}^d}M_rg(x)Rf(x)dx\\
&\lesssim& p'\int_{\mathbb{R}^d}M_rg(x)Rf(x)dx.
\end{eqnarray*}
H\"older's inequality, along with Lemma \ref{lem4.1}, tells us that
\begin{equation*}
\begin{split}
\int_{\mathbb{R}^d}M_rg(x)Rf(x)dx&\lesssim\Big[\int_{\mathbb{R}^d}\big[M_rg(x)\big]^{p'}(M_tw(x))^{1-p'}dx\Big]^{\frac{1}{p'}}\|Rf\|_{L^p(\mathbb{R}^d,M_tw)}\\
&=\Big[\int_{\mathbb{R}^d}\big[M(|g|^r)(x)\big]^{\frac{p'}{r}}\Big(M_{t\frac{p'/r-1}{p'-1}}(w^{\frac{p'-1}{p'/r-1}})(x)\Big)^{1-p'/r}dx\Big]^{\frac{1}{p'}}\\
&\lesssim\Big[\big(\frac{p'}{r}\big)'\big(t\frac{p'/r-1}{p'-1}\big)'^{\frac{r}{p'}}\Big]^{\frac{1}{r}}
\|g\|_{L^{p'}(\mathbb{R}^d,\,w^{1-p'})}.
\end{split}\end{equation*}
Combining the estimates above leads to that
\begin{equation*}
\sum_{Q\in \mathcal{S}}\|f\|_{L(\log L)^{\beta}}\langle |g|\rangle_{Q,\,r}|Q|\lesssim p'^{1+\beta}\big(\frac{p'}{r}\big)'\big(t\frac{p'/r-1}{p'-1}\big)'^{\frac{1}{p'}}\|g\|_{L^{p'}(\mathbb{R}^d,\,w^{1-p'})}.
\end{equation*}
This, via homogeneity, implies our required estimate and completes the proof of Lemma \ref{lem4.4}.
\end{proof}

Let $U$ be an operator on $\cup_{p\geq 1}L^p(\mathbb{R}^d)$. We say that $U$ is sublinear, if for all functions $f_1,\,f_2$ and $x\in\mathbb{R}^d$,
$$|U(f_1+f_2)(x)|\leq |U(f_1)(x)|+|U(f_2)(x)|,$$
and for all $\lambda\in\mathbb{R}$ and function $f$,$$|\lambda Uf(x)|=|U(\lambda f)(x)|.$$

\begin{theorem}\label{thm4.1}Let  $\alpha,\,\beta\in \mathbb{N}\cup\{0\}$, $t,\,r\in [1,\,\infty)$, $p_1\in (1,\,r')$ such that
$t\frac{p_1'/r-1}{p_1'-1}>1$. Let $U$ be a  sublinear operator which enjoys a $(L(\log L)^{\beta},\,L^r)$-sparse domination with bound $D$.
Then for any weight $u$ and bounded function $f$ with compact support,
\begin{eqnarray}\label{ine:3.2}
&&u(\{x\in\mathbb{R}^d:\,|Uf(x)|>1\})\\
&&\quad\lesssim \Big(1+ \Big\{Dp_1'^{1+\beta}\big(\frac{p_1'}{r}\big)'\big(t\frac{p_1'/r-1}{p_1'-1}\big)'^{\frac{1}{p_1'}}\Big\}^{p_1}\Big)\nonumber\\
&&\qquad\times\int_{\mathbb{R}^d}|f(y)|\log^{\beta}({\rm e}+|f(y)|)M_tu(y)dy.\nonumber
\end{eqnarray}
\end{theorem}
\begin{proof}
Let $f$   be  a bounded function with compact support, and $\mathcal{S}$ be the sparse family such that
for $g\in L^r_{{\rm loc}}(\mathbb{R}^d)$,
\begin{eqnarray*}\Big|\int_{\mathbb{R}^d}Uf(x)g(x)dx\Big|\leq D\mathcal{A}_{\mathcal{S};\,L(\log L)^{\beta},\,L^r}(f,\,g).\end{eqnarray*} By the one-third trick (see \cite[Lemma 2.5]{hlp}), there exist dyadic grids $\mathscr{D}_1,\,\dots,\,\mathscr{D}_{3^d}$ and sparse families $\mathcal{S}_1,\,\dots,\,\mathcal{S}_{3^d}$, such that for $j=1,\,\dots,\,3^d$,
$\mathcal{S}_j\subset \mathscr{D}_j$, and
$$\mathcal{A}_{\mathcal{S};\,L(\log L)^{\beta},\,L^r}(f,\,g)\lesssim\sum_{j=1}^{3^d}\mathcal{A}_{\mathcal{S}_j;\,L(\log L)^{\beta},\,L^r}(f,\,g).$$

Now let  $M_{\mathscr{D}_j, L(\log L)^{\beta}}$ be the maximal operator defined by
\begin{eqnarray}\label{maximaloperator}M_{\mathscr{D}_j,L(\log L)^{\beta}}h(x)=\sup_{Q\ni x,\,Q\in\mathscr{D}_j}\|h\|_{L(\log L)^{\beta},\,Q}.\end{eqnarray}
For each $j=1,\,\dots,\,3^d$, decompose the set $\{x\in\mathbb{R}^d:\,M_{\mathscr{D}_j,L(\log L)^{\beta}}f(x)>1\}$ as
$$\{x\in\mathbb{R}^d:\,M_{\mathscr{D}_j,L(\log L)^{\beta}}f(x)>1\}=\cup_{k}Q_{jk},$$
with $Q_{jk}$ the maximal cubes in $\mathscr{D}_j$ such that $\|f\|_{L(\log L)^{\beta},\,Q_{jk}}>1$. We have that
$$1 <\|f\|_{L(\log L)^{\beta},\,Q_{jk}}\lesssim 2^d.$$Let
$$f_1^j(y) = f(y)\chi_{\mathbb{R}^d\backslash \cup_kQ_{jk}}(y),\,\, f_2^j(y) =
\sum_k
f(y)\chi_{Q_{jk}}(y),$$
and
$$f_3^j(y) =\sum_k\|f\|_{L(\log L)^{\beta},\,Q_{jk}}\chi_{Q_{jk}}(y).$$
It is obvious that $\|f_1^j\|_{L^1(\mathbb{R}^d)}\lesssim \|f\|_{L^1(\mathbb{R}^d)}$, $\|f_1^j\|_{L^{\infty}(\mathbb{R}^d)}\lesssim 1$ and $\|f^j_3\|_{L^{\infty}(\mathbb{R}^d)}\lesssim 1$.

Let $u$ be a weight and $p_1\in (1,\,\infty)$. It then follows from Lemma \ref{lem4.4} that
\begin{eqnarray}\label{eq4.4}
&&\mathcal{A}_{\mathcal{S};\,L(\log L)^{\beta},\,L^r}(f_1^j,\,g)\\
&&\quad\lesssim p_1'^{1+\beta}\big(\frac{p_1'}{r}\big)'\big(t\frac{p_1'/r-1}{p_1'-1}\big)'^{\frac{1}{p_1'}}
\|f_1^j\|_{L^{p_1}(\mathbb{R}^d,\,M_tu)}\|g\|_{L^{p_1'}(\mathbb{R}^d,\,u^{1-p'})}.\nonumber\end{eqnarray}
Let $E=\cup_{j=1}^{3^d}\cup_k4dQ_{jk}$ and   $\widetilde{u}(y)=u(y)\chi_{\mathbb{R}^d\backslash E}(y).$ It is obvious that
\begin{eqnarray}\label{eq4.5}&&u(E)\lesssim \sum_{j,k}\inf_{z\in Q_{jk}}Mu(z)|Q_{jk}|\lesssim \int_{\mathbb{R}^d}|f(y)|\log^{\beta}({\rm e}+|f(y)|)Mu(y)dy.
\end{eqnarray}
Moreover, by the fact that $$\inf_{y\in Q_{jk}}M_{t}\widetilde{u}(y)\approx\sup_{z\in Q_{jk}}M_{t}\widetilde{u}(z),$$
we obtain that for $\gamma\in [0,\,\infty)$,
\begin{eqnarray}\label{eq4.6}
\|f_3^j\|_{L^1(\mathbb{R}^d,\,M_{t}\widetilde{u})}
&\lesssim &\sum_{k}\inf_{z\in Q_{jk}}M_{t}\widetilde{u}(z)|Q_{jk}|\|f\|_{L(\log L)^{\beta},\,Q_{jk}}\\
&\lesssim &\int_{\mathbb{R}^d}|f(y)|\log^{\beta} ({\rm e}+|f(y)|)M_{t}u(y){\rm d}y.\nonumber
\end{eqnarray}
Let $$\mathcal{S}_j^*=\{I\in\mathcal{S}_j:\, I\cap(\mathbb{R}^d\backslash E)\not =\emptyset\}.$$
Note that if ${\rm supp}\,g\subset \mathbb{R}^d\backslash E$, then
$$\mathcal{A}_{\mathcal{S}_j, L(\log L)^{\beta_1},\,L^r}(f_2^j,\,g)=\mathcal{A}_{\mathcal{S}^*_j, L(\log L)^{\beta_1},\,L^r}(f_2^j,\,g).$$
As in the argument in \cite[pp. 160-161]{hu}, we can verify that for each fixed $I\in \mathcal{S}^*_j$,
$$
\|f_2^j\|_{L(\log L)^{\beta},\,I}\lesssim \|f_3^j\|_{L(\log L)^{\beta},\,I}.
$$
Again by Lemma \ref{lem4.4}, we have that  for $g\in L^1(\mathbb{R}^d)$ with ${\rm supp}\,g\subset \mathbb{R}^d\backslash E$,
\begin{eqnarray}\label{eq4.7}&&\mathcal{A}_{\mathcal{S}_j,L(\log L)^{\beta},\,L^r }(f_2^j,\,g)\lesssim \mathcal{A}_{\mathcal{S}_j, L(\log L)^{\beta},\,L^r}(f_3^j,\,g)\\
&&\qquad\lesssim p_1'^{1+\beta}\big(\frac{p_1'}{r}\big)'\big(t\frac{p_1'/r-1}{p_1'-1}\big)'^{\frac{1}{p_1'}}\|f_3^j\|^{\frac{1}{p_1}}_{L^{1}(\mathbb{R}^d,\,M_tu)}\|g\|_{L^{p_1'}(\mathbb{R}^d\setminus E,\,u^{1-p'})}.\nonumber
\end{eqnarray}
Inequalities  (\ref{eq4.4}) and (\ref{eq4.7}) tell  us that
\begin{eqnarray*}
&&\sup_{\|g\|_{L^{p_1'}(\mathbb{R}^d\backslash E, \widetilde{u}^{1-p_1'})}\leq 1}\Big|\int_{\mathbb{R}^d}Uf(x)g(x)dx\Big|\\
&&\quad\lesssim D\sup_{\|g\|_{L^{p_1'}(\mathbb{R}^d\backslash E, \widetilde{u}^{1-p_1'})}\leq 1}\sum_{j=1}^{3^d}\Big(\mathcal{A}_{\mathcal{S}_j,\,L(\log L)^{\beta},L^r}(f_1^j,g)+\mathcal{A}_{\mathcal{S}_j,\,L(\log L)^{\beta},L^r}(f_2^j,g)\Big)\\
&&\quad\lesssim D p_1'^{1+\beta}\big(\frac{p_1'}{r}\big)'\big(t\frac{p_1'/r-1}{p_1'-1}\big)'^{\frac{1}{p_1'}}\Big(\|f_1^j\|^{\frac{1}{p_1}}_{L^{1}(\mathbb{R}^d,\,M_tu)}+
\|f_3^j\|^{\frac{1}{p_1}}_{L^{1}(\mathbb{R}^d,\,M_tu)}\Big).
\end{eqnarray*}
Thus together with inequalities (\ref{eq4.5}) and (\ref{eq4.6}), we know that
\begin{eqnarray*}
&&\quad u(\{x\in\mathbb{R}^d:\,|Uf(x)|>1\})\leq u(E)+\|Uf\|_{L^{p_1}(\mathbb{R}^d\backslash E,\,\widetilde{u})}^{p_1}\\
&&\lesssim \Big(1+ \Big\{Dp_1'^{1+\beta}\big(\frac{p_1'}{r}\big)'\big(t\frac{p_1'/r-1}{p_1'-1}\big)'^{\frac{1}{p_1'}}\Big\}^{p_1}\Big)\int_{\mathbb{R}^d}|f(y)|\log^{\beta}({\rm e}+|f(y)|)M_tu(y)dy.\nonumber
\end{eqnarray*}
This completes the proof of Theorem \ref{thm4.1}.\end{proof}

\begin{corollary}\label{cor1}Let  $\alpha,\,\beta\in \mathbb{N}\cup\{0\}$ and $U$ be a  sublinear operator. Suppose that for any  $r\in (1,\,3/2]$, $U$ satisfies bilinear $(L(\log L)^{\beta},\,L^r)$-sparse domination with bound $r'^{\alpha}$.
Then for any $w\in A_1(\mathbb{R}^d)$ and bounded function $f$ with compact support,
\begin{equation*}
\begin{split}
w(\{&x\in\mathbb{R}^d:\, |Uf(x)|>\lambda\})\\
&\quad\lesssim [w]_{A_{\infty}}^\alpha\log^{1+\beta}({\rm e}+[w]_{A_{\infty}})[w]_{A_1}\int_{\mathbb{R}^d}\frac{|f(x)|}{\lambda}\log ^{\beta}\Big({\rm e}+\frac{|f(x)|}{\lambda}\Big)w(x)dx.
\end{split}
\end{equation*}
\end{corollary}
\begin{proof}
Let $w\in A_1(\mathbb{R}^d)$. Choose $t=1+\frac{1}{2^{11+d}[w]_{A_{\infty}}}$, $r=(1+t)/2$ and $p_1=1+\frac{1}{\log ({\rm e}+[w]_{A_{\infty}})}$. We apply Theorem \ref{thm4.1} and deduce that $t\frac{p_1'/r-1}{p_1'-1}>1$ and
$$\big(t\frac{p_1'/r-1}{p_1'-1}\big)'=\frac{t(p_1'/r-1)}{t(p_1'/r-1)-p_1'+1}=\frac{t}{t-1}\frac{\frac{p_1'}{r}-1}{\frac{p_1'}{1+t}-1}
=t'\frac{\frac{2p_1'}{1+t}-1}{\frac{p_1'}{1+t}-1}\leq 5t'.
$$Note that $r'=\frac{t+1}{t-1}\le 2^{12+d} [w]_{A_{\infty}}$ and $p_1'\lesssim \log ({\rm e}+[w]_{A_{\infty}})$.
Therefore,
\begin{eqnarray*}\Big\{r'^{\alpha}p_1'^{1+\beta}\big(\frac{p_1'}{r}\big)'\big(t\frac{p_1'/r-1}{p_1'-1}\big)'^{\frac{1}{p_1'}}\Big\}^{p_1}&\lesssim &\big(r'^{\alpha}p_1'^{1+\beta}\big)^{p_1}t'^{p_1-1}\\
&\lesssim&[w]_{A_{\infty}}^{\alpha}\log ^{1+\beta}({\rm e}+[w]_{A_{\infty}}).
\end{eqnarray*}
On the other hand, we know from Lemma \ref{lem4.3} that $M_tw(y)\lesssim [w]_{A_1}w(y)$. This,   via inequality (\ref{ine:3.2}) (with $u=w$)  yields
\begin{equation*}
\begin{split}w(\{&x\in\mathbb{R}^d:\,|Uf(x)|>1\})\\
&\quad\lesssim [w]_{A_1}[w]_{A_{\infty}}^{\alpha}\log ^{1+\beta}({\rm e}+[w]_{A_{\infty}})\int_{\mathbb{R}^d}|f(y)|\log^{\beta}({\rm e}+|f(y)|)w(y)dy,
\end{split}\end{equation*}
which completes the proof of Corollary \ref{cor1}.\end{proof}
\vskip0.24cm

\section{decomposition of the composite operator}
We begin with  an endpoint estimate for composition of sublinear and linear operators.
\begin{theorem}\label{dingli4.1} Let   $U_1$ be a sublinear operator and $U_2$ be a linear operator on $\cup_{p\geq 1}L^p(\mathbb{R}^d)$. Suppose that the following conditions hold
\begin{itemize}
\item[\rm (i)]  $U_1$ is bounded on $L^2(\mathbb{R}^d)$ with bound $1$;
\item[\rm (ii)]  there exists a positive constant  $\beta_1$, such that for  any $\lambda>0$,
\begin{eqnarray}\label{ine:4.1}|\{x\in \mathbb{R}^d:\,|U_1f(x)|>\lambda\}|\le \int_{\mathbb{R}^d}\frac{|f(x)|}{\lambda}\log^{\beta_1}\Big({\rm e}+\frac{|f(x)|}{\lambda}\Big)dx;
\end{eqnarray}
\item[\rm (iii)] for some $q\in (1,\,\frac{3}{2}]$, $U_2$ enjoys a bilinear $(L(\log L)^{\beta_2},\,L^q)$-sparse domination with bound $1$.
\end{itemize}
Then we get that: for any $\lambda>0$,
\begin{eqnarray}\label{eq:4.1}&&|\{x\in \mathbb{R}^d:\,|U_1U_2f(x)|>\lambda\}|\lesssim\int_{\mathbb{R}^d}\frac{|f(x)|}{\lambda}\log^{1+\beta_1+\beta_2}\Big({\rm e}+\frac{|f(x)|}{\lambda}\Big)dx.\end{eqnarray}
\end{theorem}
\begin{proof} Let $f$ be a bounded function and $\mathcal{S}$ be a sparse family of cubes such that for any function $g$,
$$\Big|\int_{\mathbb{R}^n}g(x)U_2f(x)dx\Big|\le \mathcal{A}_{\mathcal{S},\,L(\log L)^{\beta_2},\,L^q}(f,\,g).$$
By the sparseness of $\mathcal{S}$, we get that
\begin{eqnarray*}
\Big|\int_{\mathbb{R}^d}g(x)U_2f(x)dx\Big|
&\lesssim&\int_{\mathbb{R}^n}M_{L(\log L)^{\beta_2}}f(y)M_qf(y)dy\\
&\lesssim&\|f\|_{L^2(\mathbb{R}^n)}\|M_qf\|_{L^2(\mathbb{R}^n)}.
\end{eqnarray*}
Therefore, $U_2$ is bounded on $L^2(\mathbb{R}^n)$ with bound $C_{q}$. On the other hand, we know from Theorem \ref{thm4.1} that  for any $\lambda>0$,
\begin{eqnarray}\label{equation4.3}
|\{x\in \mathbb{R}^d:\,|U_2f(x)|>\lambda\}|\le \int_{\mathbb{R}^d}\frac{|f(x)|}{\lambda}\log^{\beta_2}\Big({\rm e}+\frac{|f(x)|}{\lambda}\Big)dx.
\end{eqnarray}

Since $U_1$ is a sublinear operator and $U_2$ is a linear operator, $\lambda^{-1}|U_1U_2f|=|U_2U_2(\lambda^{-1}f)|$. Therefore it suffices to consider inequality (\ref{eq:4.1}) for $\lambda=1$.
Let $f$   be  a bounded function with compact support, and $\mathcal{S}$ be the sparse family such that (\ref{gongshi3.1}) holds true. Then there exist dyadic grids $\mathscr{D}_1,\,\dots,\,\mathscr{D}_{3^d}$ and sparse families $\mathcal{S}_1,\,\dots,\,\mathcal{S}_{3^d}$, such that for $j=1,\,\dots,\,3^d$,
$\mathcal{S}_j\subset \mathscr{D}_j$, and
$$\mathcal{A}_{\mathcal{S};\,L(\log L)^{\beta_2},\,L^q}(f,\,g)\lesssim\sum_{j=1}^{3^d}\mathcal{A}_{\mathcal{S}_j;\,L(\log L)^{\beta_2},\,L^q}(f,\,g).$$

Now let  $M_{\mathscr{D}_j, L(\log L)^{\beta_2}}$ be the maximal operator defined by (\ref{maximaloperator}).
For each $j=1,\,\dots,\,3^d$, decompose the set $\{x\in\mathbb{R}^d:\,M_{\mathscr{D}_j,L(\log L)^{\beta_2}}f(x)>1\}$ as
$$\{x\in\mathbb{R}^d:\,M_{\mathscr{D}_j,L(\log L)^{\beta_2}}f(x)>1\}=\cup_{k}Q_{jk},$$
with $Q_{jk}$ the maximal cubes in $\mathscr{D}_j$ such that $\|f\|_{L(\log L)^{\beta_2},\,Q_{jk}}>1$. We have that
$$1 <\|f\|_{L(\log L)^{\beta_2},\,Q_{jk}}\lesssim 2^d.$$
Let $E=\cup_{j=1}^{3^d}\cup_{k}16dQ_{jk}$, $f_1^j,\,\, f_2^j$, $f_3^j$ and $\mathcal{S}_j^*$ ($j=1,\,\dots,\,3^d$) be the same as we have done in the proof of Theorem \ref{thm4.1}. Write
$$|U_1U_2f(x)|\leq |U_1(\chi_{E}U_2f\big)(x)|+|U_1(\chi_{\mathbb{R}^d\backslash E}U_2f\big)(x)|=:{\rm I}_1f(x)+{\rm I}_2f(x).$$
Recall that $U_1$ satisfies the estimate (\ref{ine:4.1}), and by the fact \eqref{eq1.7}
$$|E|\lesssim \int_{\mathbb{R}^d}\frac{|f(x)|}{\lambda}
\log^{\beta_2}\Big({\rm e}+\frac{|f(x)|}{\lambda}\Big)dx.$$
It then follows that
\begin{eqnarray*}
&&|\{x\in\mathbb{R}^d:|{\rm I}_1f(x)|>1/2\} |\le \int_{E} |U_2f(x)|\log^{\beta_1} \big({\rm e}+|U_2f(x)|\big)dx\\
&=&\int_{0}^{\infty}\big|\{x\in E:\,|U_2f(x)|>s\}\big|d(s\log^{\beta_1}({\rm e}+s))\\
&\lesssim   &|E|+\int_{{\rm e}^{2\beta_1}}^\infty|\{x\in E: \,|U_{2}f(x)|>2s\}|d(s\log^{\beta_1}({\rm e}+s))\\
&\lesssim& |E|+\int_{{\rm e}^{2\beta_1}}^\infty|\{x\in E: |U_{2}(f\chi_{\{|f|>s\}})(x)|>s\}|d(s\log^{\beta_1}({\rm e}+s))\\
&&\quad+\int_{{\rm e}^{2\beta_1}}^\infty|\{x\in E: |U_{2}(f\chi_{\{|f|\leq s\}})(x)|>s\}|d(s\log^{\beta_1}({\rm e}+s)).
\end{eqnarray*}
We deduce  from  (\ref{equation4.3}) for $U_2$ that
\begin{eqnarray*}
&&\int_{{\rm e}^{2\beta_1}}^\infty|\{x\in E: |U_{2}(f\chi_{\{|f|>s\}})(x)|>s\}|d(s\log^{\beta_1}({\rm e}+s))\\
&&\quad\lesssim\int_{{\rm e}^{2\beta_1}}^\infty\int_{|f(x)|>s }\frac{|f(x)|}{s}\log^{\beta_2}\big({\rm e}+\frac{|f(x)|}{s}\big)dxd(s\log^{\beta_1}({\rm e}+s))\\
&& \quad\lesssim  \int_{\mathbb{R}^d}|f(x)|\log^{\beta_2}({\rm e}+ |f(x)|)\int_{{\rm e}}^{|f(x)|}\frac{1}{s}d(s\log^{\beta_1}({\rm e}+s))dx\\
&&\quad\lesssim \int_{\mathbb{R}^d}|f(x)|\log^{\beta_1+\beta_2+1}({\rm e}+ |f(x)|)dx,
\end{eqnarray*}
Trivial computations leads to that
$$
d(s\log^{\beta_1}({\rm e}+s))\lesssim \frac{1}{s^2}\log^{\beta_1}({\rm e}+s)ds,
$$
and when $s\in [{\rm e}^{2\beta_1},\,\infty)$,
\begin{eqnarray*} -d(\frac{1}{s}\log^{\beta_1}({\rm e}+s))&=&\Big[-\frac{1}{s^2}\log^{\beta_1}({\rm e}+s)-\frac{\beta_1}{s({\rm e}+s)}\log^{\beta_1-1}({\rm e}+s)\Big]ds\\
&\geq &\frac{1}{2s^2}\log^{\beta_1}({\rm e}+s)ds
\end{eqnarray*}
It  follows  from the $L^2(\mathbb{R}^d)$ boundedness of  $U_2$ that
\begin{eqnarray*}
&&\int_{{\rm e}^{2\beta_1}}^\infty|\{x\in E:\, |U_2(f\chi_{\{|f|\leq s \}})(x)|>s\}|d(s\log^{\beta_1}({\rm e}+s))\\
&&\quad\lesssim\int_{{\rm e}^{2\beta_1}}^\infty\frac{1}{s^2}\int_{|f(x)|\leq s}|f(x)|^2dxd(s\log^{\beta_1}({\rm e}+s))\\
&&\quad=\int_{\mathbb{R}^n}|f(x)|^2\int_{\max\{|f(x)|,\,{\rm e}^{2\beta_1}\}}\frac{1}{s^2}\log^{\beta_1}({\rm e}+s)ds\\
&&\quad\lesssim \int_{\mathbb{R}^d}|f(x)|\log^{\beta_1}({\rm e}+ |f(x)|)dx.
\end{eqnarray*}
Therefore, we conclude the estimate of ${\rm I}_1$ as follows
$$|\{x\in\mathbb{R}^d:\,|{\rm I}_1f(x)|>1/2\} |\lesssim  \int_{\mathbb{R}^d}|f(x)|\log^{\beta_1+\beta_2+1}({\rm e}+ |f(x)|)dx.$$

We turn our attention to term ${\rm I}_2f$.  By the $L^2(\mathbb{R}^d)$ boundedness of $U_1$, we know that
\begin{eqnarray*}|\{x\in\mathbb{R}^d:\,|{\rm I}_2f(x)|>1/2\} |&\lesssim & \int_{\mathbb{R}^d\backslash E}|U_2f(x)|^{2}dx\\
&\lesssim& \Big(\sup_{\|g\|_{L^{2}(\mathbb{R}^d\backslash E)\leq 1}}\Big|\int_{\mathbb{R}^d\backslash E}g(x)U_2f(x)dx\Big|\Big)^{2}.
\end{eqnarray*}
For $g\in L^{2}(\mathbb{R}^d\backslash E)$, we can write
\begin{eqnarray*}
\Big|\int_{\mathbb{R}^d\backslash E}g(x)U_2f(x)dx\Big|&\leq &\sum_{j=1}^{3^d}\mathcal{A}_{\mathcal{S}_j,\,L(\log L)^{\beta_2},\,L^q}(f_1^j,\,g)\\
&&+\sum_{j=1}^{3^d}\mathcal{A}_{\mathcal{S}_j,\,L(\log L)^{\beta_2},\,L^q}(f_2^j,\,g).
\end{eqnarray*}
Recall that if ${\rm supp}\,g\subset \mathbb{R}^d\backslash E$, then
$$\mathcal{A}_{\mathcal{S}_j, L(\log L)^{\beta_2},\,L^q}(f_2^j,\,g)=\mathcal{A}_{\mathcal{S}^*_j, L(\log L)^{\beta_2},\,L^q}(f_2^j,\,g)\lesssim \mathcal{A}_{\mathcal{S}_j, L(\log L)^{\beta_2},\,L^q}(f_3^j,\,g).$$
On the other hand, the sparseness of $\mathcal{S}_j$  states that
\begin{eqnarray*}
\mathcal{A}_{\mathcal{S}_j, L(\log L)^{\beta_2},\,L^q}(f_3^j,\,g)&\lesssim &\int_{\mathbb{R}^d}
M_{L(\log L)^{\beta_2}}f_3^j(x)M_qg(x)dx\\
&\lesssim &\|M_{L(\log L)^{\beta_2}}f_3^j\|_{L^2(\mathbb{R}^d)}\|M_qg\|_{L^2(\mathbb{R}^d)}\\
&\lesssim&\|f_3^j\|_{L^2(\mathbb{R}^d)}\|g\|_{L^2(\mathbb{R}^d)}.
\end{eqnarray*}
Since $\|f_1^j\|_{L^1(\mathbb{R}^d)}\le \|f\|_{L^1(\mathbb{R}^d)}$, and \begin{eqnarray*}
\|f_3^j\|_{L^1(\mathbb{R}^d)}
\lesssim \int_{\mathbb{R}^d}|f(y)|\log^{\beta_2} ({\rm e}+|f(y)|){d}y,
\end{eqnarray*}
we finally obtain that
\begin{eqnarray*}
|\{x\in\mathbb{R}^d:\,|{\rm I}_2f(x)|>1/2\} |&\lesssim &\Big(\sum_{j=1}^{3^d}(\|f_1^j\|_{L^2(\mathbb{R}^d)}+\|f_3^j\|_{L^2(\mathbb{R}^d)})\Big)^2\\
&\lesssim&\Big(\sum_{j=1}^{3^d}(\|f_1^j\|_{L^1(\mathbb{R}^d)}^{\frac{1}{2}}+\|f_3^j\|_{L^2(\mathbb{R}^d)}^{\frac{1}{2}})\Big)^2\\
&\lesssim&\int_{\mathbb{R}^d}|f(y)|\log^{\beta_2} ({\rm e}+|f(y)|){\rm d}y.
\end{eqnarray*}
Combining estimates for ${\rm I}$ and  ${\rm II}$ completes the proof of Theorem \ref{dingli4.1}.
\end{proof}

For  a linear operator $T$, we define the corresponding  grand maximal  operator $\mathscr{M}_{T,r}$  by
$$\mathscr{M}_{T,\,r}f(x)=\sup_{Q\ni x}|Q|^{-\frac{1}{r}}\|T(f\chi_{\mathbb{R}^d\backslash 3Q})\chi_{Q}\|_{L^{r}(\mathbb{R}^d)},$$
where the supremum is taken over all cubes $Q\subset \mathbb{R}^d$ containing $x$.
$\mathscr{M}_{T,\,r}$ was introduced by Lerner \cite{ler4} and is useful in establishing bilinear sparse domination of rough operator $T_{\Omega}$. Let $T_{1}$, $T_{2}$ be two linear operators.   We define the grand maximal operator $\mathscr{M}_{T_{1}T_{2}, r}^{*}$ by
$$\mathscr{M}_{T_1T_2,\,r}^{*}f(x)=\sup_{Q\ni x}\Big(\frac{1}{|Q|}\int_{Q}|T_1\big(\chi_{\mathbb{R}^d\backslash 3Q}
T_2(f\chi_{\mathbb{R}^d\backslash 9Q})\big)(\xi)|^rd\xi\Big)^{\frac{1}{r}}.$$
\begin{lemma}\label{yinli4.1}
Let $s\in [0,\,\infty)$ and $A\in (1,\,\infty)$, $S$ be a sublinear operator which satisfies that for any $\lambda>0$,
$$\big|\{x\in \mathbb{R}^d:\, |Sf(x)|>\lambda\}\big|\lesssim \int_{\mathbb{R}^d}\frac{A|f(x)|}{\lambda}\log^s\Big({\rm e}+\frac{A|f(x)|}{\lambda}\Big)dx.
$$
Then for any $\varrho\in (0,\,1)$ and cube $Q\subset \mathbb{R}^d$,$$
\Big(\frac{1}{|Q|}\int_{Q}|S(f\chi_{Q})(x)|^{\varrho}dx\Big)^{\frac{1}{\varrho}}\lesssim A\|f\|_{L(\log L)^{s},\,Q}.$$
\end{lemma}
\begin{proof}
Lemma \ref{yinli4.1} was proved essentially in \cite[p. 643]{huli}. We present the proof here mainly to make clear the bound. By homogeneity, we may assume that $\|f\|_{L(\log L)^{s},\,Q}=1$, which means that
$$\int_{Q}|f(x)|\log^s({\rm e}+|f(x)|)dx\leq |Q|.$$
A trivial computation leads to that
\begin{eqnarray*}
\int_{Q}|S(f\chi_{Q})(x)|^{\varrho}dx&=&\int_{0}^{A}|\{x\in Q:\,|S(f\chi_{Q})(x)|>t\}|t^{\varrho-1}dt\\
&&+\int_{A}^{\infty}|\{x\in \mathbb{R}^d:\,|S(f\chi_{Q})(x)|>t\}|t^{\varrho-1}dt\\
&\lesssim&|Q|A^{\varrho}+\int_{A}^{\infty}\int_{Q}\frac{A|f(x)|}{t}\log^s \Big({\rm e}+\frac{A|f(x)|}{t}\Big)dxt^{\varrho-1}dt\\
&\lesssim&|Q|A^{\varrho}.
\end{eqnarray*}
This gives the desired conclusion and completes the proof of Lemma \ref{yinli4.1}.
\end{proof}
\begin{lemma}\label{yinli4.2}Let $T_{1}$, $T_{2}$ be two linear operators. Suppose that for some $\alpha,\,\beta\in [0,\,\infty)$ and $r\in (1,\,2]$,
\begin{equation}\label{eq3.1xx}
|\{x\in\mathbb{R}^d:\,\mathscr{M}_{T_1,\,r}T_2f(x)>t\}|\lesssim \int_{\mathbb{R}^d}\frac{r^{\alpha}|f(x)|}{t}\log^{\beta} \Big(
{\rm e}+\frac{r^{\alpha}|f(x)|}{t}\Big)dx.\end{equation}
Then
$$
|\{x\in\mathbb{R}^d:\,\mathscr{M}_{T_1T_2,\,r}^*f(x)>t\}|\lesssim \int_{\mathbb{R}^d}\frac{r^{\alpha}|f(x)|}{t}\log^{\beta} \Big(
{\rm e}+\frac{r^{\alpha}|f(x)|}{t}\Big)dx.$$
\end{lemma}
\begin{proof}
Let  $\tau\in (0,\,1)$, $x\in\mathbb{R}^d$ and $Q\subset \mathbb{R}^d$ be a cube containing $x$.   We know by (\ref{eq3.1xx}) and Lemma \ref{yinli4.1} that
$$\Big(\frac{1}{|Q|}\int_Q\big[\mathscr{M}_{T_{1},\,r}T_{2}(f\chi_{9Q})(\xi)\big]^{\tau}d\xi\Big)^{\frac{1}{\tau}}\lesssim r^{\alpha} M_{L(\log L)^{\beta}}f(x).
$$
A straightforward  computation  leads to that
\begin{equation*}
\begin{split}
&\Big[\frac{1}{|Q|}\int_{Q}|T_{1}\big(\chi_{\mathbb{R}^d\backslash 3Q}T_{2}(f\chi_{\mathbb{R}^d\backslash 9Q})\big)(\xi)|^rd\xi\Big]^{\frac{1}{r}}\leq \inf_{\xi\in Q}\mathscr{M}_{T_{1},\,r}\big(T_{2}(f\chi_{\mathbb{R}^d\backslash 9Q})\big)(\xi)\\
&\quad\lesssim\Big(\frac{1}{|Q|}\int_Q\big[\mathscr{M}_{T_{1},\,r}T_{2}f(\xi)\big]^{\tau}d\xi\Big)^{\frac{1}{\tau}}
+\Big(\frac{1}{|Q|}\int_Q\big[\mathscr{M}_{T_{1},\,r}\big(T_{2}(f\chi_{9Q})\big)(\xi)\big]^{\tau}d\xi\Big)^{\frac{1}{\tau}}\\
&\quad\lesssim M_{\tau}\mathscr{M}_{T_{1},\,r}T_{2}f(x)+r^{\alpha}M_{L(\log L)^{\beta}}f(x).
\end{split}
\end{equation*}
On the other hand, we have
\begin{equation*}
\begin{split}
&\big|\{x\in\mathbb{R}^d:\,M_{\tau}\mathscr{M}_{T_{1},\,r}T_{2}f(x)>\lambda\}\big|\\
&\quad\quad\lesssim\lambda^{-1}\sup_{t\geq 2^{-1/\tau}\lambda}t|\{x\in\mathbb{R}^d:\,\mathscr{M}_{T_{1},\,r}T_{2}f(x)>t\}|\\
&\quad\quad\lesssim\int_{\mathbb{R}^d}\frac{r^{\alpha}|f(x)|}{\lambda}\log^{\beta} \Big({\rm e}+\frac{r^{\alpha}|f(x)|}{\lambda}\Big)dx,
\end{split}\end{equation*}
where the first inequality follows from inequality (11) in \cite{huli}. This, along with (\ref{eq1.7}) gives us the desired conclusion.
\end{proof}

\begin{theorem}\label{dingli4.2} Let $T_1,\,T_2$ be two linear operators, $r\in (1,\,3/2]$, $\beta_1,\,\beta_2,\,\gamma\in [0,\,\infty)$. Suppose that the following conditions hold
\begin{itemize}
\item[\rm (i)] $T_1$,  is bounded on  $L^{r'}(\mathbb{R}^d)$  with bound  $A$;
\item [\rm (ii)] for each $\lambda>0$,
\begin{eqnarray*}&&|\{x\in\mathbb{R}^d:\, |T_1T_2f(x)|>\lambda\}|\lesssim \int_{\mathbb{R}^d}\frac{A_0|f(x)|}{\lambda}\log^{\beta_1}\Big({\rm e}+\frac{A_0|f(x)|}{\lambda}\Big)dx;\end{eqnarray*}
\item [\rm (iii)] for each $\lambda>0$,
\begin{equation*}
|\{x\in\mathbb{R}^d:\, \mathscr{M}_{T_1,\,r'}T_2f(x)>\lambda\}|\lesssim \int_{\mathbb{R}^d}\frac{A_1|f(x)|}{\lambda}\log^{\beta_1}\Big({\rm e}+\frac{A_1|f(x)|}{\lambda}\Big)dx,
\end{equation*}
and
\begin{eqnarray*}
|\{x\in\mathbb{R}^d:\,\mathscr{M}_{T_2,\,r'}f(x)>\lambda\}|\lesssim \int_{\mathbb{R}^d}\frac{A_2|f(x)|}{\lambda}\log^{\beta_2}\Big({\rm e}+\frac{A_2|f(x)|}{\lambda}\Big)dx.
\end{eqnarray*}
\end{itemize}Then for a bounded function $f$ with compact support, there exists a $\frac{1}{2}\frac{1}{9^d}$-sparse family of cubes $\mathcal{S}=\{Q\}$,
and functions $H_1$ and $H_2$, such that for each   function $g$,
$$\Big|\int_{\mathbb{R}^n}H_1(x)g(x)dx\Big|\lesssim (A_0+A_1)\mathcal{A}_{\mathcal{S};\,L(\log L)^{\beta_1},\,L^r}(f,\,g),$$
$$\Big|\int_{\mathbb{R}^n}H_2(x)g(x)dx\Big|\lesssim AA_2\mathcal{A}_{\mathcal{S};\,L(\log L)^{\beta_2},\,L^r}(f,\,g),$$
and for a. e. $x\in\mathbb{R}^n$,
$$T_{1} T_2f(x)= H_1(x)+H_2(x).
$$
\end{theorem}

\begin{proof}
We will employ the argument in \cite{ler4}, together with some ideas in \cite{hu3}. For a fixed cube $Q_0$, define the  local analogy of $\mathscr{M}_{T_2,\,r'}$ and $\mathscr{M}^*_{T_{1}T_2,\,r'}$  by
$$ \mathscr{M}_{T_{2};\,r' ;\,Q_0}f(x)=\sup_{Q\ni x,\, Q\subset Q_0}|Q|^{-\frac{1}{r'}}\|\chi_{Q}T_{2}(f\chi_{3Q_0\backslash 3Q})\|_{L^{r'}(\mathbb{R}^d)},$$
and
$$\mathscr{M}_{T_{1}T_{2},\,r';\,Q_0}^{*}f(x)=\sup_{Q\ni x, Q\subset Q_0}\Big(\frac{1}{|Q|}\int_{Q}|T_{1}\big(\chi_{\mathbb{R}^d\backslash 3Q}
T_{2}(f\chi_{9Q_0\backslash 9Q})\big)(\xi)|^{r'}d\xi\Big)^{\frac{1}{r'}}$$respectively. Let $E=\cup_{j=1}^3E_j$ with
$$E_1=\big\{x\in Q_0:\, |T_{1}T_2(f\chi_{9Q_0})(x)|>DA_0\|f\|_{L(\log L)^{\beta_1},\,9Q_0}\big\},$$
$$E_2=\{x\in Q_0:\,\mathscr{M}_{T_2,\,r';\,Q_0}f(x)>DA_2\|f\|_{L(\log L)^{\beta_2},\,9Q_0}\},$$
$$E_3=\big\{x\in Q_0:\, \mathscr{M}^*_{T_{1}T_2,\,r';\,Q_0}f(x)>DA_1\|f\|_{L(\log L)^{\beta_1},9Q_0}\big\},$$
with $D$ a positive constant. Our hypothesis, vie Lemma \ref{yinli4.2}   tells us  that
$$|E|\le \frac{1}{2^{d+2}}|Q_0|,$$
if we choose $D$ large enough. Now on the cube $Q_0$, we apply the Calder\'on-Zygmund decomposition to $\chi_{E}$ at level $\frac{1}{2^{d+1}}$, and obtain pairwise disjoint cubes $\{P_j\}\subset \mathcal{D}(Q_0)$, such that
$$\frac{1}{2^{d+1}}|P_j|\leq |P_j\cap E|\leq \frac{1}{2}|P_j|$$
and $|E\backslash\cup_jP_j|=0$.  Observe that $\sum_j|P_j|\leq \frac{1}{2}|Q_0|$.
Let
\begin{eqnarray*}
G_1(x)&=&T_{1} T_2(f\chi_{9Q_0})(x)\chi_{Q_0\backslash \cup_{l}P_l}(x)\\
&&+\sum_lT_{1}\Big(\chi_{\mathbb{R}^n\backslash 3P_l}T_{2}(f\chi_{9Q_0\backslash 9P_l})\Big)(x)\chi_{P_l}(x).
\end{eqnarray*}
The facts that $P_l\cap E^c\not =\emptyset$ and $|E\backslash\cup_jP_j|=0$ imply that for any $g\in L^r(\mathbb{R}^d)$,
\begin{eqnarray}\label{eq3.9}
\qquad\Big|\int_{\mathbb{R}^d}G_1(x)g(x)dx\Big|&\lesssim &\int_{\mathbb{R}^d}|T_{1}T_2(f\chi_{9Q_0})(x)g(x)|\chi_{Q_0\backslash \cup_{l}P_l}(x)dx\\
&+&\sum_{l}\inf_{\xi\in P_l}|P_l|^{\frac{1}{r'}}\mathscr{M}_{T_1T_{2};Q_0,r'}^*f(\xi)\Big(\int_{P_l}|g(x)|^rdx\Big)^{\frac{1}{r}}\nonumber\\
&\lesssim&(A_0+A_1)\|f\|_{L(\log L)^{\beta_1},\,9Q_0}\langle|g|\rangle_{r,\,Q_0}|Q_0|.\nonumber\end{eqnarray}
Also, we define function  $G_2$ by

$$G_2(x)=\sum_lT_{1}\big(\chi_{3P_l}T_2\big(f\chi_{9Q_0\backslash 9P_l})\big)(x)\chi_{P_l}(x).
$$
For each function $g$, we have by H\"older's inequality that
\begin{eqnarray}\label{eq3.11}
&&\Big|\int_{\mathbb{R}^d}G_{2}(x)g(x)dx\Big|\\
&&\quad\le A\sum_l\Big(\int_{3P_l}\big|T_2\big(f\chi_{9Q_0\backslash 9P_l}\big)(x)\big|^{r'}dx\Big)^{\frac{1}{r'}}\Big(\int_{P_l}|g(y)|dy\Big)^{\frac{1}{r}}\nonumber\\
&&\quad\lesssim A\sum_l|P_l|^{\frac{1}{r'}}\inf_{\xi\in P_l}\mathscr{M}_{T_2,\,r';\,Q_0}f(\xi)\Big(\int_{P_l}|g(y)|dy\Big)^{\frac{1}{r}}\nonumber\\
&&\quad\lesssim AA_2\|f\|_{L(\log L)^{\beta_2};\,9Q_0}\langle |g|\rangle_{r,\,Q_0}|Q_0|.\nonumber
\end{eqnarray}
It is obvious that
$$T_{1} T_2(f\chi_{9Q_0})(x)\chi_{Q_0}(x)= G_1(x)+G_2(x)+\sum_{l}T_{1 }T_2(\chi_{9P_l})(x)\chi_{P_l}(x).$$

As in \cite{hu3}, we now repeat the argument above with $T_{1}T_2(f\chi_{9Q_0})(x)\chi_{Q_0}$ replaced by each $T_{1}T_2(\chi_{9P_l})(x)\chi_{P_l}(x)$, and so on.
Let $\{Q_{0}^{j_1}\}=\{P_{j}\}$,  and for fixed $j_1,\,\dots,\,j_{m-1}$, $\{Q_{0}^{j_1...j_{m-1}j_m}\}_{j_m}$ be the cubes obtained at the $m$-th stage of the decomposition process to the cube $Q_{0}^{j_1...j_{m-1}}$.
For each fixed $j_1\dots,j_m$, define the functions  $H_{Q_0,1}^{j_1\dots j_m}f$ and  $H_{Q_0,2}^{j_1\dots j_m}f$ by
\begin{eqnarray*}
&&H_{Q_0,1}^{j_1\dots j_m}f(x)=T_{1}\big(\chi_{\mathbb{R}^n\backslash 3Q_0^{j_1\dots j_m}}T_{2}(f\chi_{9Q_0^{j_1\dots j_{m-1}}\backslash 9Q_0^{j_1\dots j_m}})\big)(x)\chi_{Q_0^{j_1\dots j_m}}(x),
\end{eqnarray*}
and
$$H_{Q_0,2}^{j_1\dots j_m}f(x)=T_{1}\Big(\chi_{3Q_{0}^{j_1\dots j_m}}T_2(f\chi_{9Q_{0}^{j_1\dots j_{m-1}}\backslash 9Q_{0}^{j_1\dots j_m}})\big)\Big)(x)\chi_{Q_{0}^{j_1\dots j_m}}(x),
$$ respectively.
Set $\mathcal{F}=\{Q_0\}\cup_{m=1}^{\infty}\cup_{j_1,\dots,j_m}\{Q_{0}^{j_1\dots j_m}\}$. Then $\mathcal{F}\subset \mathcal{D}(Q_0)$ is a $\frac{1}{2}$-sparse  family.  Let
\begin{eqnarray*}&&H_{Q_0,\,1}(x)=T_{1}T_2(f\chi_{9Q_0})\chi_{Q_0\backslash \cup_{j_1}Q_{0}^{j_1}}(x)\\
&&\quad+\sum_{m=1}^{\infty}\sum_{j_1,\dots,j_m}T_{1}T_2(f\chi_{9Q_{0}^{j_1\dots j_m}})\chi_{Q_{0}^{j_1\dots j_m}\backslash \cup_{j_{m+1}}Q_{0}^{j_1\dots j_{m+1}}}(x)\\
&&\quad+\sum_{m=1}^{\infty}\sum_{j_1,\dots, j_m}H_{Q_0,1}^{j_1\dots j_m}f(x)\chi_{Q_{0}^{j_1\dots j_m}}(x).
\end{eqnarray*}
Also, we define the function $H_{Q_0, \,2}$ by
$$H_{Q_0,\,2}(x)=\sum_{m=1}^{\infty}\sum_{j_1\dots j_m}H_{Q_0,2}^{j_1\dots j_m}f(x)\chi_{Q_{0}^{j_1\dots j_m}}(x).$$
Then for a. e. $x\in Q_0$,
$$T_{1} T_2(f\chi_{9Q_0})(x)=H_{Q_0,\,1}(x)+H_{Q_0,\,2}(x).
$$
Moreover, as in the inequalities (\ref{eq3.9}) and (\ref{eq3.11}), the process of producing $\{Q_0^{j_1\dots j_m}\}$ leads to that
\begin{equation}\label{localestimate1}\Big|\int_{\mathbb{R}^d}H_{Q_0,\,1}f(x)\chi_{Q_0}(x)dx\Big|\lesssim (A_0+A_1)\sum_{Q\in\mathcal{F}}\|f\|_{L(\log L)^{\beta_1},\,9Q}\langle |g|\rangle_{r,\,Q}|Q|,\end{equation}
and for the function $g$,
\begin{equation}\label{localestimate2}
\Big|\int_{\mathbb{R}^d}g(x)H_{Q_0,\,2}(x)dx\Big|\lesssim AA_2\sum_{Q\in \mathcal{F}}|Q| \|f\|_{L(\log L)^{\beta_2},\,9Q}\langle|g|\rangle_{r,\,Q}.\end{equation}

We can now conclude the proof of Theorem \ref{dingli4.2}. In fact, as in \cite{ler4}, we decompose $\mathbb{R}^d$ by cubes $\{R_l\}$, such that ${\rm supp}f\subset 3R_l$ for each $l$, and $R_l$'s have disjoint interiors.
Then for a. e. $x\in\mathbb{R}^d$,
\begin{eqnarray*}T_{1 } T_2f(x)=\sum_{l}H_{R_l,\,1}f(x)+ \sum_lH_{R_l,\,2}f(x):= H_1f(x)+H_2f(x).\end{eqnarray*}
Our desired conclusion  follows from inequalities (\ref{localestimate1}) and (\ref{localestimate2}) directly.
\end{proof}
\vskip0.24cm

\section{Proof of Theorems}
Applying Theorem \ref{dingli4.1} to the rough singular integral operators $T_{\Omega_1}$ and $T_{\Omega_2}$, we get the following result.
\begin{corollary}\label{c41} Let $\Omega_1$, $\Omega_2$ be homogeneous of degree zero, have mean value zero and $\Omega_1,\,\Omega_2\in L^{\infty}({S}^{d-1})$. Let $r\in (1,\,3/2]$.
Then for each bounded function $f$ with compact support, there exists a $\frac{1}{2}\frac{1}{9^d}$-sparse family of cubes $\mathcal{S}=\{Q\}$,
and functions $J_1$ and $J_2$, such that for each   function $g$,
$$\Big|\int_{\mathbb{R}^d}J_1(x)g(x)dx\Big|\lesssim r'\mathcal{A}_{\mathcal{S};\,L\log L,\,L^r}(f,\,g),$$
$$\Big|\int_{\mathbb{R}^d}J_2(x)g(x)dx\Big|\lesssim r'^{2}\mathcal{A}_{\mathcal{S};\,L^1,\,L^r}(f,\,g),$$
and for a. e. $x\in\mathbb{R}^d$,
$$T_{\Omega_1} T_{\Omega_2}f(x)= J_1(x)+J_2(x).
$$
\end{corollary}
\begin{proof} Let $r\in (1,\,3/2]$.
Lerner \cite{ler4} proved that if $\Omega\in L^{\infty}(S^{d-1})$, then
\begin{eqnarray}\label{gongshi4.3}
\|\mathscr{M}_{T_{\Omega},\,r'}f\|_{L^{1,\,\infty}(\mathbb{R}^d)}\lesssim r'\|\Omega\|_{L^{\infty}(S^{d-1})}\|f\|_{L^1(\mathbb{R}^d)}.
\end{eqnarray}
On the other hand, since $T_{\Omega}$ is bounded on $L^{r'}(\mathbb{R}^d)$ with bound $\max\{r,\,r'\}$, we deduce  that
$$\mathscr{M}_{T_{\Omega},\,r'}f(x)\leq M_{r'}T_{\Omega}f(x)+\max\{r,\,r'\}M_{r'}f(x).$$
Therefore,  $\mathscr{M}_{T_{\Omega},\,r'}$ is bounded on $L^{2r}(\mathbb{R}^d)$ with bound $Cr'$. This, via estimate (\ref{gongshi4.3}), leads to that
\begin{eqnarray}\label{gongshi4.4}
\|\mathscr{M}_{T_{\Omega},\,r'}f\|_{L^{2}(\mathbb{R}^d)}\lesssim r'\|f\|_{L^2(\mathbb{R}^d)}.
\end{eqnarray}

We now conclude the proof of Corollary \ref{c41}. Let $I$ be the identity operator. It is obvious that $\mathscr{M}_{I,\,r'}T_{\Omega_2}f(x)=0$. Applying  (\ref{gongshi4.3}) and Theorem \ref{dingli4.2} with $T_1=I$, $T_2=T_{\Omega_2}$, we know that $T_{\Omega_2}$ satisfies a $(L^1,\,L^r)$-bilinear sparse domination with bound $r'$.
Thus by Theorem \ref{dingli4.1} with the fact $T_{\Omega_1}$ is of weak type (1,1) (see e.g. \cite{se}), we have that for any $\lambda>0$,
$$|\{x\in\mathbb{R}^d:\,|T_{\Omega_1}T_{\Omega_2}f(x)|>\lambda\}|\lesssim\int_{\mathbb{R}^d}\frac{|f(x)|}{\lambda}\log \Big({\rm e}+\frac{|f(x)|}{\lambda}\Big)dx.$$ Furthermore, it follows from Theorem \ref{dingli4.1}, (\ref{gongshi4.3}) and (\ref{gongshi4.4}) that for any $\lambda>0$,
$$|\{x\in\mathbb{R}^d:\,\mathscr{M}_{T_{\Omega_1},\,r'}T_{\Omega_2}f(x)>\lambda\}|\lesssim\int_{\mathbb{R}^d}\frac{r'|f(x)|}{\lambda}\log \Big({\rm e}+\frac{r'|f(x)|}{\lambda}\Big)dx.$$
Recall that $T_{\Omega_1}$ is bounded on $L^r(\mathbb{R}^d)$ with bound $Cr'$. Another application of Theorem \ref{dingli4.2} yields desired conclusion.
\end{proof}

{\it \textbf{Proof of Theorem \ref{thm1.1}}}.  For $p\in (1,\,\infty)$ and $w\in A_{p}(\mathbb{R}^d)$, let $\tau_{w}=2^{11+d}[w]_{A_{\infty}}$ and $\tau_{\sigma}=2^{11+d}[\sigma]_{A_{\infty}}$, $\varepsilon_1=\frac{p-1}{2p\tau_{\sigma}+1}$, and $\varepsilon_2=\frac{p'-1}{2p'\tau_w+1}$. It was proved in \cite{hu1} that \begin{eqnarray}\label{eq3.4}&&\ \mathcal{A}_{\mathcal{S};L^{1+\varepsilon_1},\,L^{1+\varepsilon_2}}(f,\,g)\lesssim [w]_{A_p}^{1/p}([\sigma]_{A_{\infty}}^{1/p}+[w]_{A_{\infty}}^{1/p'})\|f\|_{L^p(\mathbb{R}^d,\,w)}
\|g\|_{L^{p'}(\mathbb{R}^d,\sigma)}.\end{eqnarray}
Note that
$$\mathcal{A}_{\mathcal{S};L\log L,\,L^{1+\varepsilon_2}}(f,\,g)\lesssim \frac{1}{\varepsilon_1}\mathcal{A}_{\mathcal{S};L^{1+\varepsilon_1},\,L^{1+\varepsilon_2}}(f,\,g).$$
Invoking Corollary \ref{c41}   and inequality (\ref{eq3.4}), we deduce that for bounded functions $f$ and $g$,
\begin{eqnarray*}
\Big|\int_{\mathbb{R}^d}g(x) T_{\Omega_1}T_{\Omega_2}f(x)dx\Big|&\lesssim & [w]_{A_p}^{1/p}([\sigma]_{A_{\infty}}^{1/p}+[w]_{A_{\infty}}^{1/p'})[w]_{A_{\infty}}\\
&&\times\big([w]_{A_{\infty}}+[\sigma]_{A_{\infty}}
\big)\|f\|_{L^p(\mathbb{R}^d,\,w)}
\|g\|_{L^{p'}(\mathbb{R}^d,\sigma)}.\nonumber
\end{eqnarray*}
Recall that $w\in A_p(\mathbb{R}^n)$ implies that $\sigma\in A_{p'}(\mathbb{R}^n)$ and $[\sigma]_{A_{p'}}^{1/p'}=[w]_{A_p}^{1/p}$. Applying  Corollary \ref{c41}  to $T_{\Omega_2}T_{\Omega_1}$, we obtain that
\begin{eqnarray*}
\Big|\int_{\mathbb{R}^d}f(x) T_{\Omega_2}T_{\Omega_1}g(x)dx\Big|&\lesssim &[\sigma]_{A_{p'}}^{1/p'}([\sigma]_{A_{\infty}}^{1/p}+[w]_{A_{\infty}}^{1/p'})[\sigma]_{A_{\infty}}\\
&&\times\big([w]_{A_{\infty}}+[\sigma]_{A_{\infty}}
\big)\|f\|_{L^p(\mathbb{R}^d,\,w)}
\|g\|_{L^{p'}(\mathbb{R}^d,\sigma)}\\
&\lesssim &[w]_{A_{p}}^{1/p}([\sigma]_{A_{\infty}}^{1/p}+[w]_{A_{\infty}}^{1/p'})[\sigma]_{A_{\infty}}\\
&&\times\big([w]_{A_{\infty}}+[\sigma]_{A_{\infty}}
\big)\|f\|_{L^p(\mathbb{R}^d,\,w)}
\|g\|_{L^{p'}(\mathbb{R}^d,\sigma)}.
\end{eqnarray*}
Combining the last two inequalities yields desired conclusion.
\qed

\medskip

{\it \textbf{Proof of Theorem \ref{thm1.2}}}. Let $w\in A_1(\mathbb{R}^d)$. We obtain from Corollary \ref{c41} and Corollary \ref{cor1}  that
\begin{eqnarray*}\label{equation3.12}
&&w\big(\{x\in\mathbb{R}^d:\, |T_{\Omega_1}T_{\Omega_2}f(x)|>\lambda\}\big)\\
&&\quad\leq w\big(\{x\in\mathbb{R}^d:\, |J_1(x)|>\lambda/2\}\big)+u\big(\{x\in\mathbb{R}^d:\, |J_2(x)|>\lambda/2\}\big)\nonumber\\
&&\quad\lesssim [w]_{A_{\infty}}\log^{2}({\rm e}+[w]_{A_{\infty}})[w]_{A_1}\int_{\mathbb{R}^d}\frac{|f(x)|}{\lambda}\log \Big({\rm e}+\frac{|f(x)|}{\lambda}\Big)u(x)dx\nonumber\\
&&\qquad+[w]_{A_{\infty}}^2\log ({\rm e}+[w]_{A_{\infty}})[w]_{A_1}\int_{\mathbb{R}^n}\frac{|f(x)|}{\lambda}w(x)dx\nonumber\\
&&\quad\lesssim [w]_{A_{\infty}}^2\log ({\rm e}+[w]_{A_{\infty}})[w]_{A_1}\int_{\mathbb{R}^d}\frac{|f(x)|}{\lambda}\log \Big({\rm e}+\frac{|f(x)|}{\lambda}\Big)w(x)dx,\nonumber
\end{eqnarray*}
with $J_1$ and $J_2$ the functions defined in Corollary \ref{c41}. This completes the proof of Theorem \ref{thm1.2}.
\qed

\vskip0.24cm

{\bf Added in Proof}. After this paper was prepared, we learned that Li et al. \cite{lpr} established the weighted bounds for linear operators satisfying  the assumptions in Corollary \ref{cor1} with $\beta=0$, which coincides the conclusion  in Corollary \ref{cor1} for $\beta=0$. The argument in \cite{lpr} is different from the argument in the proof of Corollary \ref{cor1} and is of independent interest.

The authors would like to thank Dr. Kangwei Li for his helpful comments and suggestions.
\vskip0.24cm

\bibliographystyle{amsplain}

\end{document}